\documentclass{amsart}
\usepackage{amssymb, amsmath, latexsym, mathabx, dsfont,tikz}
\usepackage{multirow}

\renewcommand{\baselinestretch}{\baselinestretch}
\renewcommand{\baselinestretch}{1.1}
\numberwithin{equation}{section}

\newtheorem{thm}{Theorem}[section]
\newtheorem{lem}[thm]{Lemma}
\newtheorem{cor}[thm]{Corollary}
\newtheorem{prop}[thm]{Proposition}

\theoremstyle{definition}

\theoremstyle{remark}
\newtheorem{rmk}[thm]{Remark}

\numberwithin{equation}{section}

\newcommand{\gen}{\text{gen}}

\newcommand{\z}{{\mathbb Z}}

\newcommand{\Mod}[1]{\ (\mathrm{mod}\ #1)}
\newcommand{\dv}{\,|\,}
\newcommand{\df}[1]{\langle #1 \rangle}

\begin{document}
\title[Sums of three nonunit squares]{A sum of three nonunit squares of integers}

\author{Daejun Kim, Jeongwon Lee, and Byeong-Kweon Oh}

\address{Department of Mathematical Sciences, Seoul National University, Seoul 08826, Korea}
\email{goodkdj@snu.ac.kr}

\address{Department of Mathematical Sciences, Seoul National University, Seoul 08826, Korea}
\email{jungwon90@snu.ac.kr}

\address{Department of Mathematical Sciences and Research Institute of Mathematics, Seoul National University, Seoul 08826, Korea}
\email{bkoh@snu.ac.kr}
\thanks{This work was supported by the National Research Foundation of Korea (NRF-2017R1A2B4003758).}



\begin{abstract}   We say a positive integer is a sum of three nonunit squares if it is a sum of three squares of integers other than one. In this article, we find all integers which are sums of three nonunit squares
 assuming that the Generalized Riemann Hypothesis(GRH) holds. As applications, we find all integers, under the GRH only when $k=3$,  which are sums of $k$ nonzero  triangular numbers, sums of $k$ nonzero  generalized pentagonal numbers, and sums of $k$ nonzero generalized octagonal numbers, respectively for any integer $k\ge 3$.    
\end{abstract}

\maketitle

\section{Introduction}

The famous Legendre's three square theorem says that an integer $n$ is a sum of three squares, that is, the diophantine equation
$$
n=x^2+y^2+z^2 
$$
has an integer solution $x,y,z \in \z$ if and only if $n$ is not of the form $n=4^a(8b+7)$ for nonnegative integers $a$ and $b$. The set of all integers  that are sums of three squares is denoted by $\mathcal S_3$.  

 As a natural modification of the above theorem, one may ask to find all integers $n$ that are presented by a sum of  three ``nonzero"  squares, that is,   
 $n=x^2+y^2+z^2$ has an integer solution $x,y,$ and $z$ such that $xyz \ne 0$.  Hurwitz \cite{h} proved that any squares of integers except $4^a$ and $25\cdot 4^a$ for any nonnegative integer $a$ are sums of three nonzero squares. Pall \cite{p} proved that any integer $n \in \mathcal S_3$ that has an odd square factor greater than $1$ is a sum of three nonzero squares  unless $n=25\cdot 4^a$ for some nonnegative integer $a$. Note that $n$ is a sum of three nonzero squares if and only if $4n$ is a sum of three nonzero squares  for any nonnegative integer $n$.  
Hence  to find all integers that are sums of three nonzero squares, 
it suffices to determine the set $\mathcal S^{\circ}_3(\rm{sf})$ of all square-free integers in $\mathcal S_3$ which are not sums of three nonzero squares.
In 1959, Grosswald, Calloway, and Calloway \cite{gcc}  proved that $\mathcal S^{\circ}_3(\rm{sf})$ is a finite set.  In fact, they conjectured that 
$$
\mathcal S^{\circ}_3(\rm{sf})=\{ 1,2,5,10,13,37,58,85,130\}.
$$ 
Mordell \cite{m}  proved that for any integer $n \in \mathcal S^{\circ}_3(\rm{sf})$, the diophantine equation 
$$
xy+yz+zx=n
$$
 has a unique and specific integer solution.  By combining Mordell's characterization of the integers in $\mathcal S^{\circ}_3(\rm{sf})$ and Theorem 3.22 of \cite{cox}, we may conclude that any integer in $\mathcal S^{\circ}_3(\rm{sf})$ is, in fact,  an Euler's {\it numeri idonei} (see \cite{kani}).  Therefore, there is at most one more integer in  $\mathcal S^{\circ}_3(\rm{sf})$ other than the integers given above, and furthermore,  Grosswald, Calloway, and Calloway's conjecture is true if the Generalized Riemann Hypothesis (GRH) is true (see \cite{gkms} and \cite{kani}).  At present, it is not known whether or not  Grosswald, Calloway, and Calloway's conjecture is true without any assumption. 
 
 Now, we generalize the above results to find all integers  which are  sums of three nonzero generalized polygonal numbers.  For an integer $m\ge 3$, 
 {\it a (generalized) $m$-gonal number $P_m(x)$} is defined by
 $$
 P_m(x)=\frac{(m-2)x^2-(m-4)x}2
 $$
 for some integer $x$.  The famous Gauss's Eureka Theorem says that any positive integer is a sum of three triangular numbers, that is, for any positive integer $n$, the diophantine equation
\begin{equation}  \label{triangular}
n= \frac{x(x+1)}2+\frac{y(y+1)}2+\frac{z(z+1)}2 
\end{equation}
always has an integer solution $x,y$, and $z$.  Note that \eqref{triangular} can be  written as
$$
8n+3=(2x+1)^2+(2y+1)^2+(2z+1)^2.
$$
Hence $n$ is a sum of three nonzero triangular numbers  if and only if $8n+3$ is a sum of three squares which are not equal to $1$. Motivated by this,  we say an integer $n$ is  {\it a sum of three nonunit squares} if the following diophantine equation
$$
n=x^2+y^2+z^2 \quad  \text{and} \quad  (x^2-1)(y^2-1)(z^2-1) \ne 0
$$
  has an integer solution $(x,y,z) \in \z^3$. We define $\mathcal S_3^{\mathbf 1}$ the set of all positive integers which are sums of three nonunit squares of integers.

In this article, we prove that $\mathcal S_3-\mathcal S_3^{\mathbf 1}$ is a finite set. Moreover, we prove that
$$
\begin{array} {ll}
(\mathcal S_3-\mathcal S_3^{\mathbf 1}) \cap \{n : n\!\!\!&\equiv 0,\pm1 \Mod 5\}\\
\!\!\!&=\{1,5,6,10,11,14,19,21,26,30,35,46,51,91,235\} ,
\end{array}
$$
and under the assumption that the GRH is true, we prove that 
$$
(\mathcal S_3-\mathcal S_3^{\mathbf 1})  \cap \{n : n\equiv \pm2 \Mod 5\}=\{ 2,3,37,42,163\}.
$$
From this, one may easily deduce that under the GRH,  any integer $n$ is a sum of three nonzero triangular numbers,  
except for  $n=1,2,4,6,11,20$, and $29$. We also prove that any positive integer $n$ is a sum of $k$ nonzero triangular numbers, except for $n=1,2,\ldots, k-1$, $k+1$, and $k+3$ for any integer $k\ge 4$ without any assumption. 

For the pentagonal case, note that $n$ is a sum of three nonzero generalized pentagonal numbers if and only if the diophantine equation
$$
24n+3=x^2+y^2+z^2,  \  \  xyz \not \equiv 0 \Mod 3, \ \  \text{and} \ \ (x^2-1)(y^2-1)(z^2-1) \ne 0
$$
has an integer solution $(x,y,z) \in \z^3$. 
Hence if $24n+3$ is not divisible by $9$, then $n$ is a sum of three nonzero generalized pentagonal numbers if and only if $24n+3 \in \mathcal S_3^{\mathbf 1}$.  However, if $24n+3$ is divisible by $9$, we have to find an integer solution that is not divisible by $3$. By resolving this extra condition, we prove that under the GRH, any positive integer $n$ is a sum of three nonzero generalized pentagonal numbers, except for $n=1$ and $2$.    
Finally, for the octagonal case, note that $n$ is a sum of three nonzero generalized octagonal numbers if and only if the diophantine equation
$$
3n+3=x^2+y^2+z^2,  \  \  xyz \not \equiv 0 \Mod 3, \ \  \text{and} \ \ (x^2-1)(y^2-1)(z^2-1) \ne 0,
$$
has an integer solution $(x,y,z) \in \z^3$. In this case, we prove that under the GRH, any positive integer $n$ such that $3n+3 \in \mathcal S_3$ is a sum of three nonzero generalized octagonal numbers, except for  $n=1,2,5,6,8,9,13,16$, and $41$. We also find all integers that are sums of $k$ nonzero generalized pentagonal(or octagonal) numbers for any $k\ge 4$ without any assumption. 

The subsequent discussion will be conducted in the better adapted geometric language of quadratic spaces and lattices.   A $\z$-lattice $L=\z x_1+\z x_2+\dots+\z x_n$ of rank $n$ is a free $\z$-module equipped with non-degenerate bilinear form $B$ such that $B(x_i,x_j) \in \z$ for any $i,j$ with $1\le i, j \le n$. The corresponding quadratic map is defined by $Q(x)=B(x,x)$ for any $x \in L$. If $B(x_i,x_j)=0$ for any $i\neq j$, then we write $L=\langle Q(x_1),\ldots,Q(x_n) \rangle$.

For two $\z$-lattices $\ell$ and $L$, we say $\ell$ is represented by $L$ if there is a linear map $\sigma : \ell \to L$ such that 
$$
B(\sigma(x),\sigma(y))=B(x,y), \quad \text{for any $x,y \in \ell$.} 
$$
Such a linear map $\sigma$ is called  {\it an isometry} from $\ell$ to $L$. We also define $R(\ell,L)$ the set of all isometries from $\ell$ to $L$, and $r(\ell,L)=\vert R(\ell,L)\vert$. 

For a quadratic form $f(x_1,\ldots,x_n)=\sum_{1\le i,j \le n} a_{ij}x_ix_j$ $(a_{ij}=a_{ji})$ of rank $n$, the corresponding $\z$-lattice is defined by $L_f = \z x_1+\z x_2+\cdots +\z x_n$ with $B(x_i,x_j)=a_{ij}$ for any $i,j$ with $1\le i,j\le n$. Moreover, we define $r(m,f)=r(\langle m \rangle , L_f)$ for any positive integer $m$.

Any unexplained notation and terminology can be found in \cite{ki2} or  \cite{om}.

\section{Representations of integers as a sum of three nonunit squares}

In this section, we find all positive integers $n$ which are sums of three nonunit squares of integers, where $n \equiv 0, \pm1 \Mod 5$. We do not assume that the GRH is true in this section.

First, we introduce some useful lemma on the computation of local densities. For the definition on local densities, see \cite{ki2}.

\begin{lem}[Corollary 5.6.1 of \cite{ki2}]\label{localdensitysplit}
Let $p$ be a prime. Let $M=M_1 \perp M_2$ and $N$ be regular quadratic lattices over $\mathbb{Z}_p$, and $m_i=\text{rank }M_i>0$. Assume that all submodules of $N$ isometric to $M_1$ are transformed into each other by $O(N)$. Then we have
$$
\alpha_p(M,N) = ([M_1^\#:M_1]/[N:K\perp K^\perp])^{m_2}\cdot \alpha_{p}(M_1,N)\cdot\alpha_{p}(M_2,K^{\perp}),
$$
	where $K$ is a submodule of $N$ isometric to $M_1$.
\end{lem}

\begin{lem}\label{transformlemma}
	Let $p$ be an odd prime and let $\Delta_p$ be a nonsquare unit in $\mathbb{Z}_p$. Let $N\cong \df{1,1,1}$ be a $\mathbb{Z}_p$-lattice. Then, for any $\delta \in \{-1,-\Delta_p,p\}$,  all sublattices of $N$ isometric to $\df{\delta}$ are transformed into each other by $O(N)$. Moreover, we have
$$
\df{\delta}^\perp = \begin{cases} \df{1,\delta} & \text{if } \delta\in\mathbb{Z}_p^\times,\\ \df{-1,-p} & \text{otherwise} . \end{cases}
$$	
\end{lem}

\begin{proof} See Theorem 5.4.1 of \cite{ki2}.
\end{proof}

Let $p$ be an odd prime. For any integers $a$ and $b$, let 
$$
\ell_{a,b}=\z(e_1+ae_2+be_3)+\z(pe_2)+\z(pe_3)
$$ 
be a ternary  $\mathbb{Z}$-sublattice of $I_3=\z e_1+\z e_2+\z e_3$, where $\{e_i\}$ is a standard orthonomal basis for $I_3$. Then we have 
$$
\ell_{a,b} \cong \begin{pmatrix}
\varepsilon_{a,b} & ap & bp \\
ap & p^2 & 0\\
bp & 0   & p^2
\end{pmatrix},
$$
where $\varepsilon_{a,b}=a^2+b^2+1$.  Note that $[I_3:\ell_{a,b}]=p^2$ and hence $(\ell_{a,b})_q \simeq (I_3)_q$ for any prime $q \ne p$.

\begin{prop} \label{essentiallydistinctrep}
Let $\ell_{a,b}$ be a ternary $\mathbb{Z}$-lattice defined as above. Then we have
$$
\dfrac{r(\ell_{a,b},I_3)}{r(I_3,I_3)}=
\begin{cases}
3 & \text{if } \left(\frac{-\varepsilon_{a,b}}{p}\right)=1,\\
1 & \text{if } \left(\frac{-\varepsilon_{a,b}}{p}\right)=-1,\\
2 & \text{otherwise},
\end{cases}
$$
where $\left(\frac{\cdot}p\right)$is the Legendre symbol.
\end{prop}

\begin{proof}
For simplicity of notation, let $\ell=\ell_{a,b}$ and let $\varepsilon=\varepsilon_{a,b}$. Let $\delta=\delta(\varepsilon)$ be the element in $\{-1,-\Delta_p, p\}$ such that $\left(\frac{\delta}{p}\right)=\left(\frac{\varepsilon}{p}\right)$. Here, we are assuming that $\delta=p$ if $\varepsilon$ is divisible by $p$.  Let $M_1$ and $M_2$ be $\mathbb{Z}_p$-lattices such that
$$
M_1\cong\df{\delta} \quad \text{and} \quad M_2 \cong \begin{cases}\df{p^2,\delta p^2} & \text{if } \varepsilon\in\mathbb{Z}_p^\times, \\ \df{-p,-p^2} & \text{otherwise}.\end{cases} 
$$
Then one may easily check that 
$$
\ell_q \cong \begin{cases} M_1 \perp M_2 & \text{if } q=p,\\ (I_3)_q & \text{otherwise.} \end{cases}
$$
By the Minkowski-Siegel Formula, we have, for any ternary $\z$-lattice $L$, 
$$
r(L,\gen(I_3))=\dfrac{\pi^2}{\sqrt{dL}}\cdot \prod_{q<\infty} \alpha_q(L,I_3),
$$
 where  $\alpha_q$ is the local density over $\mathbb{Z}_q$.
Since the class number of $I_3$ is $1$, we have
$$
\dfrac{r(\ell,I_3)}{r(I_3,I_3)}=\dfrac{r(\ell,\gen(I_3))}{r(I_3,\gen(I_3))}=\sqrt{\dfrac{dI_3}{d\ell}}\cdot\prod_{q<\infty}\dfrac{\alpha_q(\ell,I_3)}{\alpha_q(I_3,I_3)}=\frac{1}{p^2}\cdot\dfrac{\alpha_p(\ell,I_3)}{\alpha_p(I_3,I_3)}.
$$
By Theorem 5.6.3 of \cite{ki2}, we have $\alpha_q(I_3,I_3) = 1-q^{-2}$ for any odd prime $q$.
By Lemma \ref{transformlemma}, we can apply Lemma \ref{localdensitysplit} with $N=(I_3)_p$ and $M=M_1\perp M_2$ so that we have
$$
\alpha_p(\ell,I_3)=\alpha_p(\df{\delta},I_3)\cdot \alpha_p(M_2,\df{\delta}^\perp).
$$
By Theorem 3.1 of \cite{y}, we have
\[\alpha_p(\df{\delta},I_3)=\begin{cases}1+\left(\frac{-\delta}{p}\right)p^{-1}  & \text{if } \varepsilon\in\mathbb{Z}_p^\times, \\
1-p^{-2} & \text{otherwise,} \end{cases}\]
and by Remark after Proposition 2 of \cite{ki1} and Theorem 5.6.3 of \cite{ki2}, we have
$$
\alpha_p(M_2,\df{\delta}^\perp)=\dfrac{\alpha_p(M_2,\df{\delta}^\perp)}{\alpha_p(\df{\delta}^\perp,\df{\delta}^\perp)}\cdot\alpha_p(\df{\delta}^\perp,\df{\delta}^\perp)=c(\varepsilon)\cdot p^2\left(1+\left(\frac{-\delta}{p}\right)p^{-1}\right),
$$
where $c(\varepsilon)=3,1$, or $2$ if $\left(\frac{-\varepsilon}{p}\right)=1,-1$, or $0$, respectively. The proposition follows from this.\end{proof}

Now, we explain  how to apply Lemma \ref{essentiallydistinctrep} to find all integers that are sums of three nonunit squares for some special case. 
Let $p$ be an odd prime. Assume that $n \in \mathcal S_3$ is an integer such that $\left(\frac{-n}p\right)=1$. Suppose that  
$$
n=1+(a+kp)^2+(b+sp)^2,
$$ 
where $0 \le a,b\le p-1$ and $k, s \in \z$. Note that possible integers $a,b$ are finite.
Then $n$ is represented by $\ell_{a,b}$. Assume that 
$$
\begin{pmatrix}
1 & 0 & 0 \\
a & p & 0\\
b & 0   & p
\end{pmatrix},
\qquad
S=(s_{ij}), \qquad \text{and} \qquad T=(t_{ij})
$$
are all representatives for the orbits under the action 
$$
O(I_3)\times R(\ell_{a,b},I_3) \to R(\ell_{a,b},I_3).
$$ 
Then we have 
\begin{align*}
n&=1+(a+kp)^2+(b+sp)^2\\
&=(s_{11}+s_{12}k+s_{13}s)^2+(s_{21}+s_{22}k+s_{23}s)^2+(s_{31}+s_{32}k+s_{33}s)^2\\
&=(t_{11}+t_{12}k+t_{13}s)^2+(t_{21}+t_{22}k+t_{23}s)^2+(t_{31}+t_{32}k+t_{33}s)^2.
\end{align*}
If $(s_{i2},s_{i3})$ and $(t_{j2},t_{j3})$ are linearly independent for any possible $i,j=1,2,3$, then  $n$ is a sum of three nonunit squares except for $36$ integers corresponding to $k,s$ satisfying  
$$
s_{i2}k+s_{i3}s=\pm1-s_{i1} \quad \text{and} \quad  t_{j2}k+t_{j3}s=\pm1-t_{j1}.
$$
As an application of this argument, we prove the following theorem.
\begin{thm}  \label{4case}
Any  integer $n \in\mathcal S_3$ with $n \equiv 4 \Mod{5}$ is a sum of three nonunit squares, except for $n=14$ and $19$.
\end{thm}

\begin{proof}
Since  we are assuming that $n \in \mathcal S_3$, there are integers $a,b$, and $c$ such that $n=a^{2}+b^{2}+c^{2}$.
If all of integers $a^{2}, b^{2}$, and $c^{2}$ are not $1$, then we obtain the desired result.
Thus, without loss of generality, we may assume that $a^{2} = 1$. Then we have $b^{2} \equiv c^{2} \equiv 4 \Mod{5}$ and we may assume that there are integers $k$ and $s$ such that $b=2+5k$ and $c=2+5s$, by changing signs, if necessary. Hence $n$ is represented by $\ell_{2,2}=\z(e_1+2e_2+2e_3)+\z(5e_2)+\z(5e_3)$.
 
Now, by Proposition \ref{essentiallydistinctrep}, we have $\dfrac{r(\ell_{2,2},I_{3})}{r(I_{3},I_{3})}=3$. Indeed, 
we may take 
$$
\begin{pmatrix} 1 & 0 & 0 \\ 2 & 5 & 0 \\ 2 & 0 & 5 \end{pmatrix}, 
\qquad
\begin{pmatrix} 1 & 4 & 0 \\ 2 & 3 & 0 \\ 2 & 0 & 5 \end{pmatrix},
\qquad \text{and} \qquad
\begin{pmatrix} 1 & 0 & 4 \\ 2 & 5 & 0 \\ 2 & 0 & 3 \end{pmatrix}
$$
as representatives for the orbits of $R(\ell_{2,2},I_3)$ under $O(I_3)$-action. Therefore, we have
\begin{align}
n&=1+(2+5k)^2+(2+5s)^2\\  \label{541}
&=(1+4k)^2+(2+3k)^2+(2+5s)^2\\ \label{542}
&=(1+4s)^2+(2+5k)^2+(2+3s)^2. 
\end{align}
If $k\ne -1,0$, then \eqref{541} implies that $n\in \mathcal{S}_3^{\mathbf 1}$, and if $s\ne -1,0$, then \eqref{542} implies that $n\in \mathcal{S}_3^{\mathbf 1}$. If $(k,s)=(0,0)$, then $9=0^2+0^2+3^2\in \mathcal{S}_3^{\mathbf 1}$. If $(k,s)=(-1,-1)$, then $n=14$, and if $(k,s)=(-1,0)$ or $(0,-1)$, then $n=19$. One may easily check that both $14$ and $19$ are not sums of three nonunit squares.  
\end{proof}

\begin{lem}\label{lemsum2}
Let $x$ be an integer with $x \equiv \pm 2 \Mod{5}$.   If $x \ne \pm 2, \pm 3$, 
then $x^{2}+1$ can be written as a sum of two nonunit squares.
\end{lem}

\begin{proof}
Since $x \equiv \pm 2 \Mod{5}$, $x$ can be written as $5 y \pm 2$ for some integer $y$, and
$$
x^{2}+1=(5y \pm 2)^{2}+1 = 5 (5y^{2} \pm 4y +1).
$$
 Assume that $5y^{2} \pm 4y +1$ is not a power of $2$. Since any odd prime factor of $1+x^{2}$ is congruent to $1$  modulo $4$, we have $r(x^2+1,I_2)>8$, that is,  $x^{2}+1$ can be written as a sum of two nonunit squares.
 Since $5y^{2} \pm 4y +1 \equiv 1,2 \Mod 4$, it is a power of $2$ only  when $y=0$ or $1$.
\end{proof}

\begin{thm} \label{0case}
Any integer $n \in \mathcal S_3$ with $n \equiv 0 \Mod{5}$ is a sum of three nonunit squares, except for $n=5,10,30,35$, and $235$.
\end{thm}

\begin{proof} Suppose that $n\in\mathcal{S}_3-\mathcal{S}_3^\mathbf{1}$. Then there are integers $a$ and $b$ such that 
$n=1+a^{2}+b^{2}$. Since we are assuming that $n$ is divisible by $5$,
 we may also assume, without loss of generality,  that $a^{2} \equiv 0 \Mod{5},  b^{2} \equiv 4 \Mod{5}$.
Therefore, there are integers $k$ and $s$ such that $n=1+ (5k)^{2} +(2+5s)^{2}$.

Note that $5k$ is not a unit.
By Lemma \ref{lemsum2}, $(2+5s)^{2}+1$ can be written as a sum of two nonunit squares if $s \ne -1,0$.
Therefore, it is sufficient to consider the case when $n$ is essentially uniquely written as a sum of three squares.
From \cite{bg}, all such integers  $n$ are $5,10,30,35,70,115,190$, and $235$. 
Among these, note that 
$$
70=3^{2}+5^{2}+6^{2},  \  \  115=3^{2}+5^{2}+9^{2}, \ \  \text{and}  \ \ 190=3^{2}+9^{2}+10^{2}.
$$ 
In fact, all the other integers except these three are not sums of three nonunit squares.
\end{proof}

\begin{lem}\label{soleqn}
For $i= 0$ or $1$ and $a=1, 4$, or $9$, all positive integer solutions of the  equation
$2^{i}5^{n} = a+y^2$ are
$$
(n,y) = \left\{ 
\begin{array}{ll}
 (1, 2) &\mbox{if  } i=0,\  a=1, \\
 (1, 1),(3,11) &\mbox{if  } i=0,\  a=4, \\
 (2,4) &\mbox{if  } i=0,\  a=9, \\
 (1,3),\  (2,7)&\mbox{if  } i=1,\  a=1, \\
 (1,1),\  (5,79)&\mbox{if  } i=1,\  a=9. 
\end{array}
\right.
$$
\end{lem}

\begin{proof}
Since all the other cases can be treated in a similar manner,
we only provide the proof of the case when $i=0$ and $a=1$.

Suppose that the equation $5^{n} = 1+y^2$ has a positive integer solution $(n,y)=(n_{0}, y_{0})$.
If we take the integer $\alpha$ such that $n_{0} \equiv \alpha \Mod{3}$ and $ 0 \le \alpha \le 2 $, then
$(x,y)=(5^{\frac{n_{0}-\alpha}{3}}, y_{0})$ is an integer solution of the elliptic curve $y^{2} = 5^{\alpha}x^{3}-1$.
Now, by using MAGMA, one may easily show that all integral points of $y^{2} = 5^{\alpha}x^{3}-1$ ($ 0 \le \alpha \le 2 $)
are $(x,y)=(1, 0)$ when $\alpha=0$, and $(x,y)=(1, \pm 2)$ when $\alpha=1$.
Therefore, $(n,y)=(1,2)$ is the only positive integer solution of the equation $5^{n} = 1+y^2$.  This completes the proof. \end{proof}

For any positive integer $k$, we define 
$$
r_k(n)=\#\{ (x_1,x_2,\dots,x_k) \in \z^k : n=x_1^2+x_2^2+\dots+x_k^2\}.
$$ 

\begin{lem}\label{change}
For any integers $a$ and $b$ such that $a^2+b^2 \notin \{1,2,5,8,18,250\}$,
there are integers $x$ and $y$ satisfying the following three properties:
\begin{itemize}
\item [(i)] $(5a)^2+(5b)^2=x^2+y^2$;
\item [(ii)] $x^2\ge10$ and $y^2 \ge10$;
\item [(iii)] $xy \not\equiv 0 \Mod{5}$.
\end{itemize} 
\end{lem}

\begin{proof}
Let  $t,u$ be integers such that $a^{2}+b^{2} = 5^{t}u$ and $(u,5)=1$.
For an integer $n$, we define
$$
\tilde{r}_{2}(n) = \# \{ (x,y) \in \mathbb{Z}^{2} \ : \ x^{2} +y^{2} =n, \ xy \not\equiv 0  \Mod{5} \}.
$$
Then we have
\begin{align*}
\tilde{r}_{2}(25a^{2}+25b^{2}) &= r_{2}(25a^{2}+25b^{2})-r_{2}(a^{2}+b^{2})
 \\
&=4  \sum_{d|25a^{2}+25b^{2}}\left(\frac{-4}{d}\right) 
-4  \sum_{d|a^{2}+b^{2}}\left(\frac{-4}{d}\right) \\
&= 8 \sum_{d|u}\left(\frac{-4}{d}\right) \ge 8.
\end{align*}
Suppose that $\alpha^{2}+x^{2}=\beta^{2}+y^{2}$ for $\alpha$, $\beta \in \{1,2,3 \}$ with $\alpha \ne\beta$.
Then one may easily show that $x=\pm \beta$ and $y=\pm \alpha$.
Therefore, for any $n>10$, we have 
$$
\# \{ (x,y) \in \mathbb{Z}^{2} \ : \ x^{2} +y^{2} =n, \ 0<x^{2} < 10, \mbox{ or } 0<y^{2} < 10 \} \le 8.
$$
Hence if $\sum_{d|u}\left(\frac{-4}{d}\right) \ge 2$,
then there are integers $x$ and $y$ satisfying all  properties given above.

Now, assume that $\sum_{d|u}\left(\frac{-4}{d}\right) =1$.
Then one may easily show that $u$ is of the form $u=2^{w}q_{1}^{2f_1}\cdots q_{s}^{2f_s}$, 
where $q_{i}$'s are primes congruent to $3$ modulo $4$, and $w$, $f_{i}$'s are nonnegative integers. Note that $u$ is of the form $u=u_0m^2$, where $u_0=1$ or $2$, and $m$ is a positive integer.

First, assume that  $m=1$,  then by  Lemma \ref{soleqn},
any integer solution $(x,y)$ of $x^2+y^2=5^{t+2}u_0=(5a)^2+(5b)^2$ satisfies
$x^2 \ge 10$ and $y^2 \ge 10$, except for the cases when
\begin{equation} \label{excep-25}
(t,u_0)=(0,1),(0,2),(1,1),(3,2).
\end{equation}
In the exceptional cases, one may easily show that there does not exist an integer solution satisfying all the three properties given above.

Now, we consider the general case.   If $(x,y)=(a,b)$ is an integer solution of $x^2+y^2=5^{t+2}u_0$,  then  $(x,y)=(ma,mb)$ is an integer solution of $x^2+y^2=5^{t+2}u$. Therefore, it suffices to consider the cases when $(t,u_0)$ satisfies \eqref{excep-25}
and $2 \le m \le 3$. Note that 
\begin{align*}
&5^{2} \cdot 2^{2}= 6^{2}+8^{2},  \  \ 5^{2} \cdot 3^{2} = 9^{2}+12^{2},  \ \ 5^{3} \cdot 2^{2} = 4^{2}+22^{2}, \ \ 5^{3} \cdot 3^{2} = 6^{2}+33^{2},  \\ 
&5^{5} \cdot 2\cdot 2^2= 6^{2}+158^{2}, \  \  \text{and}  \ \  5^{5}\cdot 2\cdot 3^{2}=9^{2}+237^{2}.
\end{align*}
Hence if $(t,u_0)=(0,1),(1,1)$, or $(3,2)$, and $m=2$ or $3$, then there is an integer solution satisfying all the properties given above. If $(t,u_0)=(0,2)$ and $m=2$ or $3$, then one may easily check that there does not exist an integer solution satisfying those properties.  This completes the proof.
\end{proof}

\begin{thm} \label{1case}
Any integer $n \in \mathcal S_3$ with $n \equiv 1 \Mod{5}$ is a sum of three nonunit squares, except for  $n=1,6,11,21,26,46,51$, and $91$.
\end{thm}

\begin{proof}
Suppose that $n\in \mathcal{S}_3-\mathcal{S}_3^{\mathbf 1}$. Then  
there are integers $a$ and $b$ such that $n=1+a^{2}+b^{2}$.
Since we are assuming that $n \equiv 1 \Mod{5}$,  we may further assume, without loss of generality,  that 
$$
a \equiv b \equiv 0 \Mod{5} \quad  \text{or} \quad a \equiv 1 \Mod{5}, \ b \equiv 2 \Mod{5}
$$
 after changing  signs of $a$ and $b$, if necessary.

First, assume that $a \equiv b \equiv 0 \Mod{5}$.
It follows from Lemma \ref{change} that if $a^{2}+b^{2}$ is not equal to $0,25,50,125,200,450$, and $6250$,
then there are integers $a_1$ and $b_1$ such that $a_1^2,b_1^2 \ge 10$, $a_1b_1 \not\equiv 0 \Mod{5}$ and $a^{2}+b^{2}=a_1^{2}+b_1^{2}$. 
Without loss of generality we may assume that $a_1$ is congruent to $\pm 2$ modulo $5$.
Since  $a_1^2 \ge 10$, $a_1$ is not equal to $\pm 2$ and $\pm 3$.
Hence by  Lemma \ref{lemsum2}, $a_1^{2}+1$ can be written as a sum of two nonunit squares, and therefore $n\in\mathcal{S}_3^{\mathbf 1}$. For the exceptional cases, note that
$$
126 = 3^{2}+6^{2}+9^{2},  \  201 = 4^{2}+4^{2}+13^{2}, \    451 = 9^{2}+9^{2}+17^{2},  \   6251 = 9^{2}+29^{2}+73^{2}.
$$
Therefore, if $n=1+a^{2}+b^{2}$ with $a \equiv b \equiv 0 \Mod{5}$, then we have $n\in\mathcal{S}_3^{\mathbf 1}$, unless $n=1,26$, or $51$. 

Now, assume that $a \equiv 1 \Mod{5}, \ b \equiv 2 \Mod{5}$.
Then there are integers $k$ and $s$ such that $a = 1+5k$ and $b=2+5s$.  Hence we have
\begin{equation} \label{pre}
n=1+(1+5k)^{2}+(2+5s)^{2}= (1+5k)^2+(2+3s)^{2}+(1+4s)^{2}.
\end{equation}
From the expression in the right hand side of \eqref{pre}, we have $n\in \mathcal{S}_3^{\mathbf 1}$ if $k\ne 0$ and $s\ne 0, -1$. 

Assume that $k=0$. If $s \equiv 1 \Mod 5$, then both $(2+3s)$ and $(1+4s)$ are divisible by $5$, which was already considered.   Therefore, we may assume that $s \not \equiv 1 \Mod 5$.  
If $s\ne 0$ and $-1$, then by Lemma \ref{lemsum2},  either $1+(2+3s)^2$ or $1+(1+4s)^2$ 
is a sum of two nonunit squares, and therefore we have $n\in \mathcal{S}_3^{\mathbf 1}$. Note that both $(k,s)=(0,0)$ and $(0,-1)$ are exceptional cases.  

Assume that $s=0$. Note that 
\begin{equation} \label{pre2}
n=1+2^2+(1+5k)^2=1+(2+4k)^2+(1-3k)^2.
\end{equation}
If $k\equiv 2 \Mod 5$, then both $(2+4k)$ and  $(1-3k)$ are divisible by $5$, which was already considered. If $k \not \equiv 2 \Mod 5$, and $k \ne -1,0,1$, then by Lemma \ref{lemsum2},  either $1+(2+4k)^2$ or $1+(1-3k)^2$ 
is a sum of two nonunit squares, and therefore we have $n\in \mathcal{S}_3^{\mathbf 1}$. Note that both $(k,s)=(0,0)$ and $(-1,0)$ are exceptional cases, and $1^2+2^2+6^2=0+5^2+4^2\in \mathcal{S}_3^{\mathbf 1}$. 

Since the proof of the case when $s=-1$ is quite similar to the above, the proof is left to the reader. 
\end{proof}

\begin{thm} \label{finite}
The number of  positive integers $n \in \mathcal S_3$ which are not  sums of three nonunit squares is finite. 
\end{thm}

\begin{proof} 
Suppose that $n \in \mathcal S_3-\mathcal{S}_3^{\mathbf 1}$ is a positive integer. Then $n$ is square-free. Define
$$
r_3^{\bullet}(n)=\#\{ (x_1,x_2,x_3) \in \z^3 : n=x_1^2+x_2^2+x_3^2, \  (x_1^2-1)(x_2^2-1)(x_3^2-1)\ne 0\}.
$$
One may easily check that 
$$
r_{3}^{\bullet}(n) = r_{3}(n) -6r_{2}(n-1)+12r_{1}(n-2).
$$
It is well known that  $r_2(n-1) \in O(n^{\epsilon})$ for any $\epsilon>0$, and if $n \not \equiv 7 \Mod 8$, then
$$
r_3(n)= \begin{cases}
\frac{16}{\pi} \sqrt{n} L(1, \chi)& \text{if } n\equiv 3 \Mod{8},\\
\frac{24}{\pi} \sqrt{n} L(1, \chi)& \text{otherwise,} \end{cases}
$$
where $\chi(\cdot) = \left( \frac{-4n}{\cdot} \right)$ and $L(1, \chi) = \sum_{m=1}^{\infty} \chi(m) m^{-1}$.
Since $L(1, \chi)^{-1} \in O(n^{\epsilon})$ for any $\epsilon>0$,
we have $r_{3}(n)>6r_{2}(n-1) $ for any sufficiently large square-free integer $n \not \equiv 7 \Mod8$, which implies that $r_{3}^{\bullet}(n)>0$ for any sufficiently large square-free integer $n \not \equiv 7 \Mod 8$. This completes the proof.
\end{proof}

\section{When $n$ is congruent to $2$ or $3$ modulo $5$}

Let $f$ be a quadratic form and let $\gen(f)$ be the genus of $f$. The set of all integers that are represented by $f$ or $\gen(f)$ is denoted by $Q(f)$ or $Q(\gen(f))$, respectively. 
 For an integer $n$, we define
$$
w(f)=\sum_{[g] \in \gen(f)/\sim} \frac1{o(g)} \qquad \text{and} \qquad r(n,\gen(f))=\frac1{w(f)}\sum_{[g] \in \gen(f)/\sim} \frac{r(n,g)}{o(g)},
$$ 
where $\gen(f)/\!\!\sim$ is the set of isometry classes $[g] \subset \gen(f)$, and $o(f)$ is the order of the isometry group $O(f)$.
The Minkowski-Siegel formula says that 
$r(n,\gen(f))$ is the product of local densities (for details, see \cite{ki2}). 

Let us consider the following two quadratic forms
$$
f(x,y,z)= 3x^{2}+25y^{2}+25z^{2}-10xy-10xz, \ \ 
g(x,y,z)=2x^{2}+25y^{2}+25z^{2}-10xy.
$$
The genus of $f$ consists of two isometry classes, and in fact,  $\gen(f)/\!\!\sim =\{[f],[g]\}$.

\begin{lem}\label{eligible}
For any positive integer $n$, we have the following:
\begin{itemize}
\item [(i)] if $n$ is square-free, then $n\in Q(\gen(f))$ if and only if $n \not \equiv 7 \Mod {8}$ and $n\equiv 2$ or $3 \Mod{5}$;
\item [(ii)] if $n \equiv 2 \Mod{5}$ and $n\in Q(f)$, then $n\in \mathcal{S}_3^{\mathbf 1}$;
\item [(iii)] if $n \equiv 3 \Mod{5}$ and $n\in Q(g)$, then $n\in \mathcal{S}_3^{\mathbf 1}$.
\end{itemize}
\end{lem}
\begin{proof}
The first assertion can be deduced by a direct computation (see \cite{OM1}). For the second assertion,
suppose that there are integers $a,b$, and $c$ such that $n=f(a,b,c)$. Then we have $3a^2 \equiv 2 \Mod{5}$, that is, $a \equiv \pm 2\Mod{5}$. Since
$$
n=f(a,b,c)= a^2+(a-5b)^2 +(a-5c)^2,
$$
we have $n\in\mathcal{S}_3^\mathbf{1}$. 
The third assertion can be proved in a similar manner if we use the fact that $g(a,b,c)=a^2+(a-5b)^2+(5c)^2$.
\end{proof}

\begin{thm}\label{GRH23}
Assume that the Generalized Riemann Hypothesis (GRH) for all Dirichlet $L$-functions and the Hasse-Weil $L$-functions of all quadratic twists of the elliptic curve $y^2+xy+y=x^3+x^2-3x+1$. Then we have the following.
\begin{itemize}
\item [(i)] The quadratic form $f$ represents all positive integers $n$ with $n\equiv2\Mod{5}$ that are represented by the genus of $f$, except for the integers of the form $16^tn_0$, where $t$ is a nonnegative integer and $n_0$ is an integer in the set 
$$
\{2,37,42,97,142,262,277,427,562,667,982,1642,3067,3502,4537,12307\}.
$$
\item [(ii)] The quadratic form $g$ represents all positive integers $n$ with  $n\equiv3\Mod{5}$  that are represented by the genus of $g$, except for the integers of the form $16^tn_0$, where $t$ is a nonnegative integer and $n_0$ is an integer in the set
$$
\{3,133,163,478,883\}.
$$
\end{itemize}

\end{thm}
Before proving the theorem, we introduce two interesting corollaries:

\begin{cor}  \label{uuu} Under the GRH, 
any integer $n \in \mathcal S_3$ with $n\equiv \pm2 \Mod{5}$ is a sum of three nonunit squares, except for $n=2,3,37,42$, and $163$.
\end{cor} 
\begin{proof}
We may assume that $n$ is a square-free. Then the corollary follows directly from Lemma \ref{eligible} and Theorem \ref{GRH23}.
\end{proof}

The following corollary was conjectured by Sun (see Remark 5.2 of \cite{S}).

\begin{cor}
Under the GRH, any positive integer $n$ is a sum of three generalized heptagonal(7-gonal) numbers, except for $n=10,16,76$, and $307$.
\end{cor}
\begin{proof}
Note that $n$ is a sum of three generalized heptagonal numbers if and only if the diophantine equation 
$$
n=\frac{5x^2-3x}{2}+\frac{5y^2-3y}{2}+\frac{5z^2-3z}{2},
$$
has an integer solution, which could be written as
$$
40n+27=(10x-3)^2+(10y-3)^2+(10z-3)^2.
$$
Note that it has an integer solution  if and only if $f(x,y,z)=40n+27$ has an integer solution.
Therefore, the corollary follows directly from Theorem \ref{GRH23}.
\end{proof}

In order to prove Theorem \ref{GRH23}, we use the similar argument which was used in \cite{l} to prove the regularities of several ternary quadratic forms under the GRH.  Let $\theta_f(z)$ be the theta series associated to $f$ which defined by 
$$
\theta_f(z)=\sum_{(x,y,z)\in\mathbb{Z}^3}q^{f(x,y,z)}=\sum_{n=0}^\infty r(n,f)q^n,\quad q=e^{2\pi iz}.
$$
Also, let $\theta_g(z)$ be the theta series associated to $g$. It is well known that they are weight $3/2$ modular forms of  level $100$ and character $\chi_{100}$, where $\chi_d(\cdot)=\left(\frac{d}{\cdot}\right)$ for any nonzero integer $d$. We put 
$$
E(z)=\sum_{n=0}^\infty r(n,\gen(f))q^n=\frac{2}{5}\theta_f(z)+\frac{3}{5}\theta_g(z).
$$
It is well known that $E(z)$ is an Eisenstein series of weight $3/2$ (for this, see \cite{SP}) and the differences
\begin{equation}\label{decomptheta}
\theta_f(z)-E(z)=-\frac{6}{5}\phi(z) \quad \text{and}\quad \theta_g(z)-E(z)=\frac{4}{5}\phi(z)
\end{equation}
are cusp forms, where
$$
\begin{array}{rl}
\phi(z)=&\hspace{-5pt}\displaystyle
\frac{1}{2}\left(\theta_g(z)-\theta_f(z)\right)=\frac{1}{2}\sum_{n=1}^\infty (r(n,g)-r(n,f))q^n=\sum_{n=1}^\infty a(n)q^n\\[15pt] 
=&\hspace{-5pt}q^2-q^3+q^8-q^{12}+2q^{13}-q^{17}-2q^{18}-3q^{22}+q^{27}+\cdots.
\end{array}
$$
Moreover, since both $f$ and $g$ are in the same spinor genus,  $\phi(z)$ is orthogonal to the space generated by unary theta functions by Satz 4 of \cite{SP}. Thus, the following Shimura lift $\Phi(z)$ of $\phi(z)$ 
$$
\Phi(z)=\sum_{n=1}^\infty A(n)q^n = q+q^2-q^3+q^4-q^6-2q^7+q^8-2q^9-3q^{11}+\cdots
$$
is a weight $2$ cusp form of level $50$.   Note that $\Phi(z)$ is the newform associated to the rational elliptic curve $E$ with Cremona label 50b1, which is given by the Weierstrass equation
$$
E:y^2+xy+y=x^3+x^2-3x+1.
$$

\begin{lem}\label{nsqfree}
Let $n$ be a positive integer in  $Q(\gen(f))$ such that $(n,5)=1$ and $\mathrm{ord}_2(n)\le1$. If $n$ is not square-free, then $n$ is represented by both $f$ and $g$.
\end{lem}
\begin{proof} Since both $f$ and $g$ are in the same spinor genus and the class number of $f$ is two, there is a quadratic form $f' \in [f]$ adjacent to $g$ in the graph $\z(f,p)$ defined in \cite{sp0} for any prime $p$ not dividing $10$ (see also \cite{bh}).  The lemma follows directly from this. 
\end{proof}

\begin{cor}\label{nrepset}
Let $S_f=\{n\in Q(\gen(f))-Q(f) : n\equiv 2 \Mod{5}, \ n\text{ is square-free}\}$ and $S_g=\{n\in Q(\gen(g))-Q(g) : n\equiv 3 \Mod{5}, \ n\text{ is square-free}\}$.
We have
$$
\left(Q(\gen(f))-Q(f)\right) \cap \{n : n\equiv 2 \Mod{5}\} = \{16^tn_0 : n_0 \in S_f\}
$$
and
$$
\left(Q(\gen(g))-Q(g)\right) \cap \{n : n\equiv 3 \Mod{5}\} = \{16^tn_0 : n_0 \in S_g\}.
$$

\end{cor}
\begin{proof}
One may easily show that $n\in Q(f)$ or $Q(g)$ if and only if $16n\in Q(f)$ or $Q(g)$, respectively. The corollary follows from this and Lemma \ref{nsqfree}.
\end{proof}

Thus, in order to prove Theorem \ref{GRH23}, it is enough to determine $S_f$ and $S_g$ in Corollary \ref{nrepset}. Hence from now on, we only consider square-free integers in the genus of $f$.

\begin{lem}\label{rngenf}
Let $n \in Q(\gen(f))$ be a square-free  positive integer. Then we have 
$$
r(n,\gen(f))=b_n\sqrt{n}L(1,\chi_{-100n}),
$$
where $b_n = \frac{4}{3\pi}$ if $n\equiv 3\Mod{8}$, and $b_n=\frac{2}{\pi}$ otherwise.
\end{lem}
\begin{proof}
By the Minkowski-Siegel formula, we have 
$$
r(n,\gen(f))=\dfrac{2\pi}{25}{\sqrt{n}}\cdot \prod_{p<\infty} \alpha_p(n,f),
$$
where $\alpha_p$ is the local density over $\mathbb{Z}_p$.  By using  \cite{y}, one may easily check that
$$
\alpha_p(n,f)=\begin{cases}
1+\left(\frac{-n}{p}\right)p^{-1}& \text{if } p\ne 2,5 \text{ and } p \not\mid n,\\
1-p^{-2}&\text{if } p\ne2,5 \text{ and } p \dv n,\\
1& \text{if } p=2 \text{ and } n\equiv 3\Mod{8},\\
3/2& \text{if } p=2 \text{ and } n\not\equiv 3\Mod{8},\\
2& \text{if } p=5.
\end{cases}
$$
Therefore, we have 
$$
\begin{array}{rl}
r(n,\gen (f))\hspace{-8pt}&=\displaystyle\frac{4\pi}{25}\sqrt{n}\cdot\alpha_2(n,f)
\prod_{\substack {p\neq 2,5 \\ p\not\mid n}} \left(1+\left(\frac{-n}{p} \right)\frac{1}{p}\right)\cdot
\prod_{\substack {p\neq 2,5 \\ p\mid n}} \left(1-\frac{1}{p^2}\right)\\
&=\displaystyle\frac{4\pi}{25}\sqrt{n}\cdot\alpha_2(n,f)L(1,\chi_{-100n})
\prod_{p\neq 2,5} \left(1-\frac{1}{p^2}\right)\\
&=b_n\sqrt{n}L(1,\chi_{-100n}).
\end{array}
$$
This completes the proof.
\end{proof}

For a positive integer $N$ and a positive rational number $k$ such that $2k\in\mathbb{Z}$, let $S_k(N,\chi)$ be the space of cusp forms of weight $k$ with character $\chi$ for the congruence group $\Gamma_0(N)$.
To compute the growth of the Fourier coefficients $a(n)$ of $\phi(z)$, we introduce the following theorem which is a special case of the theorem of Waldspurger \cite{W}.

\begin{thm}[Waldspurger \cite{W}]\label{Waldspurger} 
Let $\phi(z)\in S_{3/2}(N,\chi_t)$ be an eigenform of each of the Hecke operators $T(p^2)$ for any $p\nmid N$ such that its Shimura lift is the newform associate to a rational elliptic curve $E$.
If $n$ and $m$ are two positive square-free integers such that $n/m\in(\mathbb{Q}_p^\times)^2$ for each $p$ dividing $N$, and $\phi(z)=\sum_{k=1}^\infty a(k)q^k$, then
$$
a(n)^2 m^{1/2}\chi_t(m/n)L(1,E(-tm)) =a(m)^2n^{1/2}L(1,E(-tn)),
$$
where $L(s,E(D))$ is the Hasse-Weil $L$-function of the $D$-quadratic twist of $E$.
\end{thm}

\begin{proof}[Proof of Theorem \ref{GRH23}] Recall that $\phi(z)=\frac12(\theta_g(z)-\theta_f(z))$. As mentioned above, 
we have  $\phi(z)\in S_{3/2}(100,\chi_{100})$ and its Shimura lift is $\Phi(z)\in S_2(50,\chi_{100}^2)$, which is the newform associate to $E$ satisfying the hypotheses of Waldspurger's Theorem. 
Let $n \in Q(\gen(f))$ be a square-free positive integer. Then there is a unique positive integer $m\in \{2,3,13,17,22,42,62\}$ such that $n/m\in (\mathbb{Q}_p^\times)^2$. 
By applying Theorem \ref{Waldspurger}, we have
\begin{equation}\label{an}
a(n)^2=\frac{a(m)^2}{m^{1/2}L(1,E(-100m))}\cdot n^{1/2}L(1,E(-100n)).
\end{equation}
Hence if $n\equiv 2\Mod{5}$ is not represented by $f$, then by combining (\ref{decomptheta})$\sim$(\ref{an}), Lemma \ref{rngenf}, and by bounding the values $L(1,E(-100m))$ for any integer $m$ given above, we have
$$
\frac{L(1,E(-100n))}{L(1,\chi_{-100n})^2}\ge c_{n,f} \cdot n^{1/2},
$$
where $c_{n,f}=0.223422$ if $n\equiv 3\Mod{8}$, and $c_{n,f}=0.502705$ otherwise. On the other hand, assuming the GRH, we may use Chandee's theorems in \cite{C} to compute that
$$
\frac{L(1,E(-100n))}{L(1,\chi_{-100n})^2}\le 51.697\cdot n^{0.18799},
$$
which implies that 
$$
n\le \begin{cases} 3.783\times 10^7 \quad \text{if $n\equiv 3\Mod{8}$},\\  
                2.813\times 10^6 \quad  \text{otherwise}.
                \end{cases}
$$                 
Similarly, if an integer $n\equiv 3\Mod{5}$ is not represented by $g$, then one may prove that $n\le 2.813\times 10^6$ if $n\equiv 3\Mod{8}$, and $n\le 2.1\times10^5$ otherwise. 
From these, one may determine, under the GRH, the sets $S_f$ and $S_g$ in Corollary \ref{nrepset} by direct computations. This completes the proof.
\end{proof}

\begin{thm} \label{octause} Under the GRH, any integer $n\in\mathcal{S}_3$ is a sum of three nonunit squares, except for
$$
n= 1,2,3,5,6,10,11,14,19,21,26,30,35,37, 42, 46,51,91,163, \  \text{and}  \  235.
$$ 
\end{thm}

\begin{proof} 
The theorem follows directly from Theorems \ref{4case}, \ref{0case}, \ref{1case}, and \ref{uuu}.
\end{proof}

\begin{rmk} {\rm In fact, it is relatively easy to find all positive integers which are sums of $k$ nonzero and nonunit squares for any $k\ge 4$. For the proofs and the lists of such integers, see \cite{kim1}. (See also \cite{kim2}.)  }
\end{rmk}

\begin{cor} Under the GRH,  any positive integer $n$ is a sum of three nonzero triangular numbers, except for $n=1,2,4,6,11,20$,  and $29$. 
\end{cor} 

\begin{proof}  Note that $n$ is a sum of three nonzero triangular numbers if and only if $8n+3\in \mathcal{S}_3^\mathbf{1}$. Hence the corollary follows directly from Theorem \ref{octause}.
\end{proof}

\begin{thm}\label{4tri}
For any integer $k \ge 4$, any positive integer $n$ is a sum of $k$ nonzero triangular numbers, except for $n=1,2,\ldots, k-1$, $k+1$, and $k+3$.
\end{thm}

\begin{proof}
First, we prove that any positive integer except for $1,2,3,5$, and $7$ is a sum of four nonzero triangular numbers. 
Let $n$ be an integer greater than or equal to $37$.
Since every positive integer is a sum of three triangular numbers, 
$n-36$ can be  written as a sum of $i$ nonzero triangular numbers
for some $i$ with $1 \le i \le 3$. Note that
$$
 36=P_{3}(8)=P_{3}(6)+P_{3}(5)=P_{3}(5)+P_{3}(5)+P_{3}(3),
$$ 
that is, 36 can be written as a sum of $4-i$ nonzero triangular numbers.
Hence $n$ is a sum of four nonzero triangular numbers.
For a positive integer $n$ with $n \le 36$, 
one may easily check that $n$ is a sum of four nonzero triangular numbers, except for $n=1,2,3,5$, and $7$.

Now, suppose that the statement of the theorem holds for a given integer $k\ge4$.
Note that if $n-1$ is a sum of $k$ nonzero triangular numbers, then $n$ is a sum of $k+1$ nonzero triangular numbers.
Also, it is obvious that if $n=1,2,\ldots, k$, $k+2$, or $k+4$, then $n$ cannot be written as a sum of $k+1$ nonzero triangular numbers. This completes the proof.
\end{proof}

\section{A sum of three nonzero pentagonal(or octagonal) numbers }

In this section, we find all integers that are sums of three nonzero generalized pentagonal(or octagonal) numbers.

\begin{lem} \label{penta-tec}
Let $m \ne 3$ be a positive integer and let $a,b,c$ be integers satisfying
$$
m=a^2+b^2+c^2, \quad abc \not \equiv 0 \Mod 3.
$$
Then there are integers $A,B,C$ such that 
$$
9m=A^2+B^2+C^2, \ \ ABC \not \equiv 0 \Mod 3, \ \ \text{and} \ \  (A^2-1) (B^2-1) (C^2-1) \ne 0.
$$ 
\end{lem}

\begin{proof}  We may assume that $a\equiv b\equiv c\equiv 1 \Mod 3$, if necessary,  by changing the signs of $a,b$, and $c$. Note that
\begin{align} \label{z1}
9m&=(a+2b+2c)^2+(-b+2c-2a)^2+(-c-2a+2b)^2 \\ \label{z2}
&=(-a-2b+2c)^2+(b+2c+2a)^2+(-c+2a-2b)^2\\ \label{z3}
&=(-a+2b-2c)^2+(-b-2c+2a)^2+(c+2a+2b)^2.
\end{align}
Furthermore, note that all of the nine terms from $(a+2b+2c)$ to $(c+2a+2b)$ in \eqref{z1}$\sim$\eqref{z3} are congruent to $2$  modulo $3$. Therefore, it suffices to show that all of the three terms in at least one of \eqref{z1}, \eqref{z2}, and \eqref{z3} are not $-1$. Suppose, on the contrary, that at least one term in all of  \eqref{z1}, \eqref{z2}, and \eqref{z3} is $-1$. Then by considering all of the possible $27$ cases, one may easily show that it happens only when $a=b=c=1$, which is a contradiction to the assumption.  
\end{proof}

\begin{thm}  \label{penta-octa}
Let $m$ be an integer such that $m=a^2+b^2+c^2$ for some integers $a,b$, and $c$. If $m \not \in \{1,2,3,14\}$, then  there are integers $x,y,z$ such that 
$$
x^2+y^2+z^2=9m, \ \ xyz \not \equiv 0 \Mod 3, \ \ \text{and} \ \ (x^2-1) (y^2-1) (z^2-1) \ne 0.
$$
\end{thm}

\begin{proof}  First, assume that $m \equiv 2\Mod 3$.   Without loss of generality, we may assume that, if necessary, by interchanging the role of $a,b$, and $c$, and by changing the signs of $b$ and $c$,  
$$
a \equiv 0 \Mod 3 \quad \text{and} \quad b\equiv c \equiv 1 \Mod 3.
$$
Note that 
\begin{align}  \label{y}
9m&=(-a+2b+2c)^2+(-b+2c+2a)^2+(-c+2a+2b)^2 \\ \label{yy}
&=(a+2b+2c)^2+(-b+2c-2a)^2+(-c-2a+2b)^2.
\end{align}
Furthermore, note that all of the six terms from $(-a+2b+2c)$ to $(-c-2a+2b)$ in  \eqref{y}$\sim$  \eqref{yy} are  congruent to $1$ modulo $3$.  If  all of the three terms in \eqref{y} or  \eqref{yy} are not one, then we are done.  Hence we may assume that at least one of the three terms in \eqref{y} and  in \eqref{yy} are one. For example, if 
$$
-a+2b+2c=1 \quad \text{and} \quad -b+2c-2a=1,
$$
then there is an integer $t$ such that 
$$
a=-6t-3, \ \ b=2t+1, \ \  \text{and}\ \  c=-5t-2.
$$
Similarly, by considering all  of the  possible $9$ cases, we may assume that there is an integer $t$ such that 
$$
(a,b,c)=(-6t-3,2t+1,-5t-2), \ \ (0,2t-1,t), \ \  \text{or} \ \  (3t,2t+1,-2t+1).
$$ 
In the first case, $t$ is divisible by $3$, for we are assuming that both $b$ and $c$ are congruent to $1$ modulo $3$. By letting $t=3u$, we have 
\begin{align}
9m&=9\cdot\left[(-18u-3)^2+(6u+1)^2+(-15u-2)^2\right] \\  \label{p}
&=(19u+2)^2+(2u+1)^2+(70u+11)^2 \\ \label{pp}
&=(10u+1)^2+(26u+5)^2+(67u+10)^2.
\end{align}
If $u \not \equiv 1 \Mod 3$, then we are done by \eqref{p}, except for the cases when $u=-1,0$. If $u=0$, then $m=14$. 
One may easily check by a direct computation that $9\cdot14$ does not have an integer solution $x,y$, and $z$ satisfying the conditions given above.  If $u=-1$, then we have 
$$
17^2+1+59^2=11^2+13^2+59^2. 
$$
If $u \equiv 1 \Mod 3$, then we are done by \eqref{pp}.

If $(a,b,c)=(0,2t-1,t)$, then we have $t\equiv 1 \Mod 3$. By letting $t=3u+1$ for some integer $u$, we have 
\begin{align}
9m&=9\cdot\left[(6u+1)^2+(3u+1)^2\right] \\  \label{a}
&=(u-1)^2+(2u+1)^2+(20u+4)^2 \\ \label{aa}
&=(4u)^2+(10u+3)^2+(17u+3)^2.
\end{align}
If $u \not \equiv 1 \Mod 3$, then all of the three terms in \eqref{a} are not divisible by $3$. Hence we are done if $u\ne -1,0,2$. If $u=0$, then  $m=2$, and 
$$
\begin{cases} 
2^2+1+16^2=4^2+7^2+14^2 \quad &\text{if $u=-1$}, \\ 
1+5^2+44^2=5^2+16^2+41^2 \quad &\text{if $u=2$.} \end{cases}
$$
Note that $9\cdot2$ does not have an integer solution $x,y$, and $z$ satisfying the conditions given above. 
If $u\equiv 1\Mod 3$, then we are done by \eqref{aa}. 
 
If $(a,b,c)=(3t,2t+1,-2t+1)$, then we have $t \equiv 0\Mod 3$. By letting $t=3u$ for some integer $u$, we have
\begin{align} 
9m&=9\cdot\left[(9u)^2+(6u+1)^2+(-6u+1)^2\right]\\  \label{q}
&= (2u-3)^2+(2u+3)^2+(37u)^2 \\  \label{qq}
&=(5u-4)^2+(14u-1)^2+(34u+1)^2.
\end{align}
If $u \not \equiv 0 \Mod 3$, then we are done  by \eqref{q},  except for the cases when $u=\pm1, \pm2$. In the exceptional cases, we have
$$
\begin{cases} 
1+5^2+37^2=7^2+11^2+35^2 \quad &\text{if $u=\pm1$}, \\ 
1+7^2+74^2=10^2+55^2+49^2 \quad &\text{if $u=\pm2$.} \end{cases}
$$  
If $u$ is divisible by $3$, then we are done by \eqref{qq}, except for the case when $u=0$, that is, $m=2$.

Now, assume that $m\equiv 1 \Mod 3$. In this case, without loss of generality, we may assume that 
$$
a\equiv b\equiv 0 \Mod 3 \quad \text{and}  \quad c\equiv 1 \Mod 3.
$$   
Note that
\begin{align} \label{r0}
9m&=(-a+2b+2c)^2+(-b+2c+2a)^2+(-c+2a+2b)^2 \\  \label{r}
&=(a+2b+2c)^2+(-b+2c-2a)^2+(-c-2a+2b)^2\\  \label{rr}
&=(-a-2b+2c)^2+(b+2c+2a)^2+(-c+2a-2b)^2.
\end{align}
Furthermore, note that  all of the nine terms from $(-a+2b+2c)$ to $(-c+2a-2b)$  in \eqref{r0}$\sim$\eqref{rr} are  congruent 2 modulo $3$.  Hence all of the three terms in \eqref{r0}, \eqref{r}, or \eqref{rr} are not $-1$, then we are done.  Suppose, on the contrary, that at least one of the three terms in each \eqref{r0}, \eqref{r}, and \eqref{rr} is $-1$. Then, by direct computations for all of the possible $27$ cases, we have $(a,b,c)=(0,0,1), (-3,-3,1)$. 
In the former case, we have $m=1$, and in the latter case, we have $m=19$. Note that $9\cdot19=5^2+5^2+11^2$.

The case when $m\equiv 0 \Mod 3$ can directly be  proved by using Lemma \ref{penta-tec} repeatedly, if necessary,  and the fact that
$$
9=1+2^2+2^2, \ 18=1+1+4^2, \ 27=1+1+5^2, \  126=1+2^2+11^2,
$$ 
that is,  all of these integers satisfy the condition given in Lemma \ref{penta-tec}. This completes the proof.
  \end{proof}

\begin{thm}  Under the GRH, 
 any positive integer $n$ is a sum of three nonzero generalized pentagonal numbers, except for $n=1$ and $2$. 
\end{thm} 

\begin{proof}  Let $n$ be an integer greater than $2$. 
 It is well known that  $n$ is a sum of three pentagonal numbers, that is, 
\begin{equation} \label{penta1}
\frac{(3x^2-x)}2+\frac{(3y^2-y)}2+\frac{(3z^2-z)}2=n
\end{equation}
always has an integer solution. 
Since Equation \eqref{penta1} is equivalent to
\begin{equation}  \label{penta2}
(6x-1)^2+(6y-1)^2+(6z-1)^2=24n+3,
\end{equation}
it suffices to show that Equation \eqref{penta2} has a nonzero integer solution $x,y,z$.  We know that by Theorem \ref{octause}, the diophantine equation $x^2+y^2+z^2=24n+3$ always has an integer solution $x,y,z$ all of whose values are not $\pm 1$.  Hence, if $24n+3$ is not divisible by $9$, then, by changing the signs of $x,y$, and $z$, if necessary, we can take an integer solution $x,y,z \ (\ne -1)$ all of whose values are congruent to $-1$ modulo $6$. This implies that Equation  \eqref{penta2} has a nonzero integer solution.  If $24n+3$ is divisible by $9$, then Equation  \eqref{penta2} also has a nonzero integer solution by Theorem \ref{penta-octa}.  This completes the proof. \end{proof}

\begin{thm}
For any integer $k \ge 4$, any positive integer $n$ is a sum of $k$ nonzero generalized pentagonal numbers, except for $n=1,2,\ldots, k-1$.
\end{thm}

\begin{proof} Note that
$$
7=P_{5}(-2)=P_{5}(2)+P_{5}(-1)=P_{5}(2)+P_{5}(1)+P_{5}(1).
$$ 
The remaining of the proof is quite similar to that of Theorem \ref{4tri}.
\end{proof}

Now, we consider the octagonal case. 
 
\begin{lem} \label{octa} An integer $n$ is a sum of three nonzero generalized octagonal numbers if and only if $3n+3$ is a sum of three nonunit squares.   
\end{lem} 

\begin{proof} Note that 
$$
(3x^2-2x)+(3y^2-2y)+(3z^2-2z)=n
$$
has a nonzero integer solution $x,y,z$ if and only if
$$
(3x-1)^2+(3y-1)^2+(3z-1)^2=3n+3.
$$
has a nonzero integer solution. Note that 
$$
9=1^2+2^2+2^2,\ 18=1^2+1^2+4^2, \ 27=1^2+1^2+5^2, \ 126=1^2+2^2+11^2.
$$
The lemma follows directly from this and Theorem \ref{penta-octa}.  
\end{proof}

\begin{thm}  Under the GRH,  any positive integer $n$ that is a sum of three generalized octagonal numbers is also a sum of three nonzero generalized octagonal numbers, except for  $n=1,2,5,6,8$, $9,13,16$,  and   $41$.
\end{thm} 
 
\begin{proof} The theorem follows directly from Theorem \ref{octause} and Lemma \ref{octa}. 
\end{proof}

 \begin{thm}
For any integer $k \ge 4$, any positive integer $n$ is a sum of $k$ nonzero generalized octagonal numbers, except for $n=1,2,\ldots, k-1$, and $k+b$, where $b \in B=\{1,2,3,5,6,9,10,13,17\}$. 
\end{thm}

\begin{proof}
First, consider the case when $k=4$. Note that 
$$
(3x^2-2x)+(3y^2-2y)+(3z^2-2z)+(3w^2-2w)=n, \quad xyzw\ne 0
$$ 
has an integer solution if and only if $X^2+Y^2+Z^2+W^2=3n+4$ has an integer solution such that 
\begin{equation}\label{eq031111}
XYZW \not\equiv 0 \Mod{3} \quad \text{and} \quad (X^2-1)(Y^2-1)(Z^2-1)(W^2-1)\neq 0.
\end{equation}
Assume that $N=3n+4>226$ and $N$ is not divisible by $4$. 
Then there is an integer $\omega \in\{2,4,5,7,8,10\}$ such that 
$$
N-\omega^2 \not \equiv 0,4,7 \Mod8 \quad \text{and} \quad N-\omega^2 \equiv 0 \Mod{9}.
$$ 
Let $m$ be an integer such that  $N-w^2=9m$. Then $m$ is also a sum of three squares. 
Furthermore, since $m=\frac{(N-\omega^2)}9>14$,  there are integers $x,y,z$ such that 
$$
x^2+y^2+z^2=9m, \ \  xyz\not\equiv0\Mod{3}, \  \  \text{and}  \ \ (x^2-1)(y^2-1)(z^2-1)\neq0,
$$
 by Theorem \ref{penta-octa}. Hence we are done in this case. For the case when $N=3n+4\le 226$ with $N$ not divisible by $4$, one may easily check that the equation $x^2+y^2+z^2+w^2=N$ has an integer solution satisfying \eqref{eq031111}, except for 
 $$N\in E=\{7,10,13,19,22,25,31,34,43,46,55,67\}$$

In order to consider the case when $N$ is divisible by $4$, we note that if the diophantine equation $x^2+y^2+z^2+w^2=M$ has an integer solution satisfying \eqref{eq031111}, then so does the diophantine equation $x^2+y^2+z^2+w^2=4M$.
Furthermore, for each integer $N_0\in E$, one may easily check that the diophantine equation $x^2+y^2+z^2+w^2=4N_0$ has an integer solution satisfying \eqref{eq031111}. Therefore, any positive integer $n$ is a sum of four nonzero generalized octagonal numbers, except for $n=1,2,3$, and $4+b$, where $b \in B$.

Now, suppose that the statement of the theorem holds for a given integer $k\ge4$.
Note that if $n-1$ is a sum of $k$ nonzero generalized octagonal numbers, then $n$ is a sum of $k+1$ nonzero generalized octagonal numbers.
Conversely,  let $n=k+b+1$ be an integer for some $b \in B$ which is a sum of $k+1$ nonzero generalized octagonal numbers.
Then there is an integer  $x_{i} \in \z-\{0\}$ such that $n=\sum_{i=1}^{k+1}P_{8}(x_{i})$.
Since $n=k+b+1 \le k+17+1 < 5(k+1)$, at least one of $x_{1},\ldots,x_k$, or $x_{k+1}$ is one.
Therefore, $n-1=k+b$ should be a sum of $k$ nonzero generalized octagonal numbers. This completes the theorem.
\end{proof}

\end{document}